\documentclass[oneside,letterpaper]{amsart}

\usepackage{hyperref}
\usepackage{enumerate}
\usepackage{graphicx}
\usepackage{amssymb,color}

\providecommand{\floor}[1]{\left\lfloor#1\right\rfloor}
\providecommand{\Z}{\mathbb{Z}}

\providecommand{\N}{\mathbb{N}}

\providecommand{\Q}{\mathbb{Q}}
\DeclareMathOperator{\rank}{rank}

\newtheorem{theorem}{Theorem}[section]
\newtheorem{lemma}[theorem]{Lemma}
\newtheorem{conjecture}[theorem]{Conjecture}
\newtheorem{claim}[theorem]{Claim}

\newtheorem{corollary}[theorem]{Corollary}
\theoremstyle{definition}
\newtheorem{definition}[theorem]{Definition}
\theoremstyle{remark}

\numberwithin{equation}{section}

\begin{document}

\title{A Rainbow Ramsey Analogue of Rado's Theorem}

\author{J. A. De Loera}
\author{R. N. La Haye}
\author{A. Montejano}
\author{D. Oliveros}
\author{E. Rold\'an-Pensado}

\address[J. A. De Loera, R. N. La Haye]{Department of Mathematics, UC Davis}
\email{deloera@math.ucdavis.edu, rlahaye@math.ucdavis.edu}
\address[A. Montejano]{Facultad de Ciencias, UNAM campus Juriquilla}
\email{amandamontejano@ciencias.unam.mx}
\address[D. Oliveros, E. Rold\'an-Pensado]{Instituto de Matem\'aticas, UNAM campus Juriquilla}
\email{dolivero@matem.unam.mx, e.roldan@im.unam.mx}

\begin{abstract}
We present a Rainbow Ramsey version of the well-known Ramsey-type theorem of Richard Rado. We use techniques from the Geometry of Numbers. We also disprove two conjectures proposed in the literature.
\end{abstract}

\maketitle

\section{Introduction}

Given a set $X$, a \emph{$k$-coloring} of $X$ is a surjective mapping $c:X\to \{1,2,...,k\}$, or equivalently a partition $X=C_1\cup C_2 \, ...\cup C_k$ into $k$ nonempty parts called \emph{color classes}. A subset $Y\subseteq X$ is called \emph{monochromatic} under $c$ if it is contained in a single color class. On the other hand, $Y$ is called \emph{rainbow} if the coloring assigns pairwise distinct colors to its elements. Given a coloring of a subset of the integers, we say that a vector is \emph{rainbow} if each of its entries is colored differently.

\emph{Arithmetic Ramsey Theory} concerns the study of the existence of monochromatic structures, like arithmetic progressions or solutions of linear equations, in every coloring of subsets of the integer numbers. The classical results in this area include Schur's Theorem: For every $k$, if $n$ is sufficiently large, every $k$-coloring of the starting segment of integers $[n]=\{1,2,...,n\}$ contains a monochromatic solution to the equation $x+y=z$. Another result is Van der Waerden's Theorem which states that for every pair of integers $t$ and $k$, when $n$ is sufficiently large, every $k$-coloring of $[n]$ contains a monochromatic $t$-term arithmetic progression. One of the most important examples is the famous 1933 theorem of Richard Rado: Given a rational matrix $A$, consider the homogeneous system of linear equations $Ax = 0$. This system or the matrix is called $k$-\emph{regular} if, for every $k$-coloring of the natural numbers, the system has a monochromatic solution. A matrix is \emph{regular} if it is $k$-regular for all $k$. Rado's Theorem characterizes precisely those matrices that are regular. The characterization depends on simple additive conditions satisfied by the columns of $A$ (which can be found in \cite{radoregular,goodbook}). In fact, Rado's Theorem is a common generalization of both Schur's and Van der Waerden's Theorems.

In contrast to Ramsey Theory, \emph{Rainbow Ramsey Theory} refers to the study of the existence of \emph{rainbow} structures in colored combinatorial universes under some density conditions on the coloring. Arithmetic versions of this theory have been recently studied by several authors concerning colorings of integer intervals or cyclic groups, showing the existence of rainbow arithmetic progressions or rainbow solutions to linear equations under some density conditions on the color classes \cite{fmr, rainbowsurvey, llm, ms, rainbowapk}. As pointed out in papers \cite{fmr, rainbowsurvey}, one natural research direction is to generalize the known monochromatic results to the case of rainbow solutions of systems of linear equations.
In particular the authors of \cite{fmr} stated that it would be very exciting to provide a complete rainbow analogue of Rado's Theorem. 
The key purpose of this paper is to provide such a theorem. As a consequence we disproved two conjectures \cite{rainbowsurvey,fmr}. Our techniques combine established combinatorial tools with ideas from convex geometry, particularly Ehrhart's Theory of lattice point counting \cite{BarviPom,beckrobins}.

\begin{definition}
A matrix $A$ with rational entries is \emph{rainbow partition $k$-regular} if for all $n$ and 
for every equinumerous $k$-coloring of $[kn]$ (i.e. $k$-colorings in which all color classes have size $n$), there exists a rainbow vector in $\ker(A)$.
The smallest $k$ such that $A$ is rainbow partition $k$-regular, if it exists, is called the \emph{rainbow number of $A$} and is denoted by $r(A)$.
A matrix $A$ is \emph{rainbow regular} if it is rainbow partition $k$-regular for all sufficiently large $k$.
\end{definition}

For instance, it is known that both matrices $$A_1=\begin{pmatrix}1 & -2 & 1 \end{pmatrix}\text{ and }A_2=\begin{pmatrix}1 & 1 & -1 & -1 \end{pmatrix},$$ corresponding to $3$-term arithmetic progressions and solutions to the Sidon equation, are rainbow regular matrices with rainbow numbers $r(A_1)=3$ and $r(A_2)=4$ respectively (see \cite{fmr, rainbowapk}).

The authors of \cite{fmr} and \cite{rainbowapk} claimed that every $1\times n$ matrix with nonzero rational entries is rainbow regular if and only if some of the entries have different signs. That is correct for $n\geq 3$, but incorrect for $n=2$. This subtle difference is key for finding the main theorem. Lemma \ref{lem:1x2} shows that all  rational nonzero $1\times2$ matrices are not rainbow regular and is used to prove Theorem \ref{thm:main}.

The papers \cite{fmr, rainbowapk} contain two different conjectures on the classification of rainbow regular matrices. However, both conjectures as originally stated have the trivial counterexample $\begin{pmatrix}1 & -1 & 0\end{pmatrix}$. To avoid this, we state them here in slightly modified versions.

\begin{conjecture}[Jungi\'c, Ne\v set\v ril, Radoi\v ci\'c \cite{rainbowsurvey}]
\label{wrongrainbow3}
A matrix $A$ with integer entries is rainbow regular if and only if there exist two linearly independent vectors with distinct positive integer entries in $\ker(A)$.
\end{conjecture}
\begin{conjecture}[Fox, Mahdian, Radoi\v ci\'c \cite{fmr}]
\label{wrongrainbow4}
A matrix $A$ with integer entries is rainbow regular if and only if the rows of $A$ are linearly independent and $\ker(A)$ contains a vector with distinct positive integer entries.
\end{conjecture}
 
After finding counterexamples to the above conjectures, we were able to obtain a Rado-style classification theorem of rainbow regular matrices. Moreover, the definition of rainbow regularity is stronger than it seems. We show that if $A$ is rainbow regular, it satisfies a stronger version of rainbow regularity where the equinumerous condition is relaxed.


\begin{definition}
A matrix $A$ with rational entries is \emph{robustly rainbow regular} if there exists some constant $C$, depending only on $A$ such that for every $\varepsilon>0$, positive integer $N$, and large enough integer $k$, the following holds:
For every $k$-coloring of $[N]$ in which each color class contains at most $(C-\varepsilon)\frac N{\sqrt k}$ elements, there is a rainbow vector in $\ker(A)$.
\end{definition}

Note that (robust) rainbow regularity is actually a property of the kernel rather that the matrix. 

\pagebreak

In our main theorem below, we classify rainbow regularity in terms of both the matrix $A$ and its kernel. 

\begin{theorem}\label{thm:main}
Let $A$ be a $m \times d$ rational matrix. The following conditions are equivalent.
\begin{enumerate}[(i)]
\item $A$ is rainbow regular.
\item $A$ is robustly rainbow regular.
\item There exists a vector in $\ker(A)$ with positive integer entries, and every submatrix of $A$ obtained by deleting two columns has the same rank as $A$.
\item There exists at least one vector in $\ker(A)$ with positive integer entries, and for every pair of distinct indices $(i,j)$, there exists a pair of vectors $x=(x_1,\dots,x_d)$ and $y=(y_1,\dots,y_d)$ in $\ker(A)$ such that $x_iy_j\neq x_jy_i$.
\end{enumerate}
\end{theorem}

From Theorem \ref{thm:main} the reader can easily see that the matrix $\begin{pmatrix}a & -b & 0\end{pmatrix}$, with $a,b$ positive integers gives a counterexample to Conjectures \ref{wrongrainbow3} and \ref{wrongrainbow4}. However, not all counterexamples are this simple. For instance, the matrix
$$\begin{pmatrix}
 1 & 0 & 1 & -1 & 0 \\
 0 & 1 & 1 & 0 & -1 \\
 1 & 0 & 0 & 1 & -1
\end{pmatrix}$$
has a kernel generated by $(1,2,3,4,5)$ and $(1,2,4,5,6)$ but is not rainbow regular.

In the next section we give the proof of Theorem \ref{thm:main} and we prove the following corollary about $k$-colorings of $\N$ instead of $[N]$. 

\begin{corollary}\label{thm:cor1}
Given a rational rainbow regular $m\times d$ matrix $A$ there exists a constant $C$, depending only on $A$, that satisfies the following:
For every $k$-coloring of $\N$ with each color class having upper density less than $\frac{C}{\sqrt k}$, there is a rainbow vector in $\ker(A)$.
\end{corollary}

Ramsey theory has also been used in graph theory, rainbow Ramsey theory is no different and we state a corollary to Theorem \ref{thm:main}
that describes the properties of graphs and it suggests an interesting connection to the theory of no-where-zero flows on graphs (see \cite{diestel,sixflow} and references therein).
 For the sake of completeness  we recall some definitions and well known facts. A graph is called   \emph{$k$-edge-connected} if has no edge cut of cardinality less that $k$. A \emph{$k$-flow} of an oriented graph $G=(V,E)$ is an integer  function $\phi$ on $E$ such that $-k< \phi(e) < k$ for all $e\in E$, and  satisfies the Kirchhoff's law, that is $\sum_{e \in \delta^+(v)} \phi(e) = \sum_{e \in \delta^-(v)} \phi(e)$ for each $v\in V$. If a graph has a $k$-flow, then it has a positive $k$-flow.  If in addition  $\phi (e)\neq 0$ for every $e\in E$ we call $\phi$ a  \emph{nowhere-zero-$k$-flow}. For a given oriented graph $G=(V,E)$ with $|V|=n$ and $|E|=m$ we consider  the incidence matrix $M$ which is a $n\times m$ matrix. The rank of $M$ is $n-c$ where $c$ is the number of connected components of $G$.  Theorem \ref{thm:main} has the following graph theoretic corollary (for a monochromatic analogue see \cite{hogben}).

\begin{corollary} \label{thm:corgraph}
The connected components of a graph $G$ are all 3-edge-connected if and only if for some orientation of $G$
there exists a constant $C$ depending only on the graph such that for every $\varepsilon>0$, positive integer $N$, and large enough integer $k$,  it follows that: for every $k$-coloring of $[N]$ in which each color class contains at most $(C-\varepsilon)\frac{N}{\sqrt k}$ elements, there is a rainbow flow on that orientation of $G$.
\end{corollary}

In the third section we look at the matrices which give rise to 
Fibonacci sequences; we use Theorem \ref{thm:main} to show that they are rainbow regular and give a bound for 
their rainbow number.

\section{Proof of Theorem \ref{thm:main} and its corollaries}
\label{arrct}

We start this section by considering the simplest case. We show that for any $1\times 2$ matrix $A$ there is an equinumerous $k$-coloring of $[kn]$, for sufficiently large $k$ and $n$, without rainbow vectors in the kernel of $A$. This case will later be used in the proof of the main theorem.

\begin{lemma}\label{lem:1x2}
If $A$ is a nonzero rational $1\times2$ matrix then $A$ is not rainbow regular.
\end{lemma}
\begin{proof}

Assume $A=\begin{pmatrix}p&q\end{pmatrix}$ with $p,q\in\Q$.
Then $\ker(A)$ is generated by some vector $(a,b)$.

If either $p$ or $q$ are equal to $0$ then either $a$ or $b$ equals $0$, thus $\ker(A)$ contains no vectors with positive integer entries. The same conclusion holds if $p$ and $q$ have the same sign. Therefore, in both cases, the matrix $A$ is not rainbow regular.

Assume $p$ and $q$ are nonzero and have opposite signs.
If $p=-q$, then $\ker(A)$ is generated by $(1,1)$ and $A$ is not rainbow regular. So we may assume that $p\neq-q$, $a$ and $b$ are relatively prime, and $a<b$.

Let $N=nk$. In order to define an equinumerous coloring without rainbow solutions, we give a partition $\mathcal P$ of $[N]$. Each of its classes will be monochromatic. So $i$ must be in the same class as $j$ if either $(i,j)$ or $(j,i)$ are in $\ker(A)$, i.e., $ai=bj$ or $bi=aj$.

Since $a$ and $b$ are relatively prime, every integer can be written uniquely as $a^\alpha b^\beta c$ where $c$ is not divisible by neither $a$ or $b$. The equivalence classes are of the form
\begin{equation*}
\left\{a^\alpha b^0 c,a^{\alpha-1} b^1 c,\dots,a^0 b^\alpha c\right\}\cap[N].
\end{equation*}

As mentioned before, all the elements in each class will be colored with the same color.
Now, using a greedy algorithm, we can define the $k$-coloring of $[N]$:
In each step, assign a least used color to a largest uncolored class. For a step-by-step example of our coloring procedure see Figure \ref{N12}.

\begin{figure}
\includegraphics{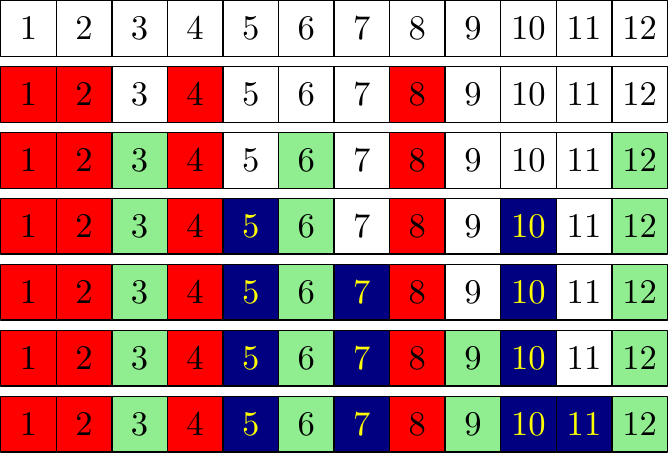}
\caption{An explicit coloring for $N=12$ showing that the matrix $\big(1\quad -2\big)$ is not rainbow partition 3-regular.}
\label{N12}
\end{figure}

All that remains is to show that this is an equinumerous $k$-coloring. To do this we need some bounds related to $\mathcal P$; specifically, on the cardinality of the largest class and the number of classes with one element.

Observe that each class can be represented by its smallest element. The set of these representative elements is precisely the set of integers in $[N]$ which are not divisible by $b$.
Note that the largest element of any class in $\mathcal P$ is of the form $a^{\gamma}b^\beta c$ with $a^{\gamma}b^\beta c\le N$, and the number of elements of this class is $\beta+1$. So clearly the maximum cardinality of any class is at most $1+\log_b(N)$.

We proceed by counting the number of classes in $\mathcal P$ with cardinality one. These classes are represented by $a^\alpha c$ with either $\alpha=0$ or $a^{\alpha-1}bc>N$. Here we will assume that both $k$ and $n$ are equal and large compared to $a$ and $b$, so we can find asymptotic bounds up to multiplicative constants depending only on $a$ and $b$.

The number of classes corresponding to $\alpha=0$ are
$$\floor{\frac{N}{a}}+\floor{\frac{N}{b}}-\floor{\frac{N}{ab}}=\Omega(N).$$
The classes corresponding to $a^{\alpha-1}bc>N$ are in bijection with the elements $c\in\left(\frac{a}{b}N,N\right]\setminus b\Z$. The number of those classes is approximately
$$\frac{b-1}{b}\left(N-\frac{a}{b}N\right)=\Omega(N).$$
Thus the number of classes with cardinality one is $\Omega(N)$.

Assume that the coloring defined above is not equinumerous. 
Then at some point during the greedy algorithm, a color becomes used more than $n$ times.
Consider the first time that this happens: some equivalence class of size $m$ is assigned a color 
that has already been used at least $n-m+1$ times. Note that $m>1$, otherwise we would have already finished coloring.
Because on each step we use the least used color, each color has been already used at least $n-m+1$ times.
It follows that $k(n-m+1)\ge N-k(1+\log_b(N))$ integers have been colored.
Therefore at most $k(1+\log_b(N))$ elements of $[N]$ remain uncolored, but we have not colored any class of size one yet. Therefore there are at least $\Omega(N)$ uncolored integers. This gives
$$\Omega(N)\le k(1+\log_b(N))=O\left(N^{1/2}\log(N)\right),$$
a contradiction.

Therefore, for sufficiently large $k=n$, the coloring is equinumerous and thus $A$ is not rainbow $k$-regular.
\end{proof}

Now we are ready to prove Theorem \ref{thm:main}. We point out that implication $(iii)\implies (ii)$ is based on a proof of \cite{rainbowapk}, but with new ideas from the theory of lattice points inside polyhedra started E. Ehrhart \cite{BarviPom,beckrobins}.

\begin{proof}[Proof of Theorem \ref{thm:main}]\mbox{}

\noindent\emph{(i)$\implies$(iv):}
We show the contrapositive. If $\ker(A)$ contains no vectors with positive integer entries or there exist indices $i\neq j$ such that  $x_iy_j=x_jy_i$ for all $x,y\in\ker(A)$, then $A$ is not rainbow regular. 

If $\ker(A)$ contains no vector with all positive integer entries then $A$ is clearly not rainbow regular, 
so we assume that there exists $q\in\ker(A)$ with positive integer entries and that for every $x\in\ker(A)$, $x_iq_j=x_jq_i$. 

So $\begin{pmatrix}q_j & -q_i\end{pmatrix}$ is a $1\times 2$ nonzero rational matrix which, by Lemma \ref{lem:1x2}, is not rainbow regular.
That is, for arbitrarily large $k_0$ there exists $k>k_0$, $n$ and an equinumerous $k$-coloring of $[kn]$ such that if $x_iq_j-x_jq_i=0$, then $x_i$ and $x_j$ share a color.
Therefore $A$ is not rainbow regular.\\

\noindent{\emph{(iv)$\implies$(iii):}}
Suppose that there are two columns $i$ and $j$ of $A$ such that the submatrix $A'$ obtained by deleting them from $A$ has a different rank than $A$.

If $A$ has linearly dependent rows, we may remove them and pass to the corresponding submatrices of $A$ and $A'$ without changing the kernel.

Since $A'$ is a submatrix of $A$, $\rank(A')<\rank(A)$. Because $A$ and $A'$ have the same number of rows but $A'$ has smaller rank, the rows of $A'$ are linearly dependent.

Let $r_\ell$ and $r'_\ell$ be the $\ell$-th rows of $A$ and $A'$, respectively. Then there exist scalars 
$\alpha_1,\dots,\alpha_m$ not all zero, such that $\sum_{\ell=1}^m\alpha_\ell r'_\ell=0$. Consequently, 
the entries of $$r=\sum_{\ell=1}^m\alpha_\ell r_\ell$$ are zero except possibly for the $i$-th and $j$-th entries.
Let $q_i$ and $q_j$ be those entries. Then every $x,y\in\ker(A)$ satisfy $r\cdot x=r\cdot y=0$, thus $q_ix_i+q_jx_j=q_iy_i+q_jy_j=0$ and therefore $x_iy_j=x_jy_i$.\\

\noindent{\emph{(iii)$\implies$(ii):}}
Since each color class contains at least one element, then $(C-\varepsilon)\frac N{\sqrt k}\ge 1$ so that the hypothesis in \emph{(ii)} can be satisfied. Consequently, we need only prove that $\ker(A)$ contains a rainbow vector for every large enough $N$ independent of $k$.

We wish to find an upper bound for the number of vectors with entries in $[N]$ that are not rainbow. For this we introduce some new notation. If $i,j\le d$ then $A_{\widehat{i,j}}$ denotes the matrix obtained by removing the $i$-th column $A_i$ and the $j$-th column $A_j$ from $A$, and if $x$ is a vector then $x_{\widehat{i,j}}$ denotes the vector obtained by removing the $i$-th and $j$-th entries $x$.

If $x\in[N]^d\cap\ker(A)$ is not rainbow, then $x$ has two entries that share a color (they may have the same value). Assume $i$ and $j$ are these entries, and $z_1$ and $z_2$ are their respective values. Then $x_{\widehat{i,j}}$ must solve the equation
\begin{equation*}\label{eq:ij}
A_{\widehat{i,j}}y=z_1A_i+z_2A_j.
\end{equation*}
In other words, if this equation has no solutions in $[N]^{d-2}$ then there is no $x\in[N]^d\cap\ker(A)$ with $x_i=z_1$ and $x_j=z_2$.
If this equation has some solution $y_0$, then the set of solutions is simply $y_0+\ker(A_{\widehat{i,j}})$.

By assumption,
$\text{rank}(A_{\widehat{i,j}}) = \text{rank}(A)$, hence $\dim(\ker(A_{\widehat{i,j}}))=\dim(\ker(A))-2$.
Thus we only need to choose values for $\dim(\ker(A))-2$ of the remaining coordinates and the rest will be determined.
Since there are at most $N$ possible values for each coordinate, there are at most $N^{\dim(\ker(A))-2}$ ways to choose the remaining $d-2$ values of $x$.

Now we can bound the number of non-rainbow vectors in $[N]^d\cap \ker(A)$.
Since each color class contains at most $(C-\varepsilon)\frac N{\sqrt k}$ elements, there are at most $$k\left((C-\varepsilon)\frac N{\sqrt k}\right)^2=(C-\varepsilon)^2N^2$$ pairs of integers in $[N]$ that share a color.
Given two such integers $z_1$ and $z_2$, there are at most $\binom{d}{2}$ ways to place them in a vector $x\in[N]^d$.
Therefore there are at most
\begin{equation}\label{eq:bound}
(C-\varepsilon)^2N^2\binom{d}{2}N^{\dim(\ker(A))-2}=(C-\varepsilon)^2\binom d2N^{\dim(\ker(A))}
\end{equation}
non-rainbow vectors $x$.

Now we must bound the number of vectors in $[N^d]\cap\ker(A)$ from below.
Consider the polytope $P=[0,1]^d\cap\ker(A)$ and its interior $P^\circ$.
The number of vectors in $[N]^d\cap\ker(A)$ is bounded below by the number of integer vectors in $NP^\circ$.

Since $P$ is a rational polytope, the number of integer points in $NP$ is given by its Ehrhart quasi-polynomial $L_P(N)$ \cite[Chapter 3]{beckrobins}, which has degree $\dim(P)=\dim(\ker(A))$.
By Ehrhart-Macdonald reciprocity \cite[Chapter 4]{beckrobins}, the number of integer points in $NP^\circ$ is given by the quasi-polynomial $p(N)=(-1)^{\dim(P)}L_P\left(-N\right)$.
Let $\nu$ be the leading coefficient of $p$ (this is in fact the volume of the polytope $P$). Note that $\nu$ is a constant depending only on $A$.

Take $C=\sqrt{\frac\nu{\binom d2}}$. Then $p(N)$ is larger than the coefficient of $N^{\dim(\ker(A))}$ in \eqref{eq:bound} for all $\varepsilon>0$. Therefore, for sufficiently large $N$, the non-rainbow vectors do not cover $[N]^d\cap\ker(A)$.\\

\noindent{\em (ii)$\implies$(i):}
This follows immediately by taking $N=kn$.
\end{proof}


%
%

Now we show that as a consequence of Theorem \ref{thm:main} that we can deal with colorings of $\N$ where the color classes have bounded upper density:

\begin{proof} (of Corollary \ref{thm:cor1})
Let $A$ be a rational rainbow regular $m\times d$ matrix. By Theorem \ref{thm:main}, there is some $C:=C(A)$ such that for every $\varepsilon>0$, positive integer $N$, and large enough integer $k$, it follows that:  every $k$-coloring of $[N]$ in which each color class contains at most $(C-\varepsilon)\frac{N}{\sqrt k}$ elements,  contains a rainbow vector in $\ker (A)$.

Suppose $\N$ is $k$-colored such that each color class has upper density strictly less than $\frac{C}{\sqrt k}$.
Then there exists $N\in\N$ and small $\varepsilon>0$ such that for each color class $K\subseteq\N$ and all $n>N$, 
$$\frac{\#(K\cap[n])}{n}\le \frac{C}{\sqrt{k}}-\varepsilon.$$
Consequently, $\ker (A)$ contains a rainbow vector.
\end{proof}

We now move to the graph theory consequences:

\begin{proof} (of  Corollary \ref{thm:corgraph}) Let $G$ be a graph.

\noindent{$\implies$:}
Suppose each connected component of $G$ is $3$-edge-connected.
Then $G$ is bridgeless so by \cite{sixflow}, $G$ has a nowhere-zero $6$-flow.
Consequently, we may choose an orientation of $G$ which has a positive $6$-flow.
Let $M$ be the incidence matrix corresponding to that orientation.
Note that the positive $6$-flow, written as a vector, is an element of the  $\ker (M)$ with positive integer entries.

Consider a submatrix obtained by deleting two columns from $M$.
This corresponds to the subgraph obtained by deleting two edges from $G$.
Because $G$ is $3$-edge connected, the deleting  of two edges from $G$ does 
not change the number of connected components---and thus rank---of $G$.
Hence any submatrix obtained by deleting two columns from $M$ has the same rank as $M$.

Therefore $M$ is (robustly) rainbow regular by Theorem \ref{thm:main}.
That is,  there exists a constant $C$ depending only on $M$ (and thus $G$) such that  for every $\varepsilon>0$, positive integer $N$, and large enough integer $k$, it follows that: for every $k$-coloring of $[N]$ in which each color class contains at most $(C-\varepsilon)\frac{N}{\sqrt k}$ elements,
there is a rainbow vector in $\ker (M)$---which corresponds to a rainbow flow on the chosen orientation of $G$.

\noindent{$\Longleftarrow$:}
Suppose that for some orientation of $G$
there exists a constant $C$ depending only on the graph such that for every $\varepsilon>0$, positive integer $N$, and large enough integer $k$, it follows that: for every $k$-coloring of $[N]$ in which each color class contains at most $(C-\varepsilon)\frac{N}{\sqrt k}$ elements, there is a rainbow flow on that orientation of $G$. Then the incidence matrix $M$ corresponding to the chosen orientation of $G$ is robustly rainbow regular.
By theorem \ref{thm:main}, the rank of any submatrix obtained by deleting two columns from $M$ is the same as the rank of $M$.
This implies that the deletion of any two edges from $G$ does not change the rank---and thus the number of connected components---of $G$.
Therefore each connected component of $G$ is $3$-edge-connected.
\end{proof}

\section{Examples}

We would like to note that as a corollary of Theorem \ref{thm:main} all known examples of rainbow regular matrices are in fact robustly rainbow regular. This includes several well-known families such as matrices associated with arithmetic progressions. 
In this section we use Theorem \ref{thm:main} to analyze Fibonacci sequences and show that their associated matrices are rainbow regular (and thus robustly rainbow regular); we also give bounds for their rainbow number.

%
%
%
%
Here we look at sequences $p_1,\dots,p_d$ where $p_{i+2}=p_i+p_{i+1}$; we call these \emph{Fibonacci sequences}. In this case we use the $(d-2)\times d$ matrices
$$A_d=\begin{pmatrix}
1 & 1 & -1 & 0 & 0 & \dots & 0 & 0 \\
0 & 1 & 1 & -1 & 0 & \dots & 0 & 0 \\
0 & 0 & 1 & 1 & -1 & \dots & 0 & 0 \\
\vdots & \vdots & \ddots & \ddots & \ddots & \ddots & \vdots & \vdots \\
0 & 0 & 0 & 0 & 0 & \dots & -1 & 0\\
0 & 0 & 0 & 0 & 0 & \dots & 1 & -1
\end{pmatrix}.$$
To verify that this matrix is rainbow regular, we show that it satisfies the condition described in part \emph{(iv)} of Theorem \ref{thm:main}.

Let $F_i$ be the usual Fibonacci sequence with $F_0=0$ and $F_1=1$. Let $x=(F_1,\dots,F_d)$ and $y=(F_2,\dots,F_{d+1})$. The vector $x$ only has positive integer entries and it is easy to see that $x_i y_j=F_iF_{j-1}\neq F_jF_{i-1}=x_j y_i$.

Next, we present exponential bounds on the rainbow number $r(A_d)$.

Recall that for $d=3$, $A_3$ is the Schur equation and $r(A_3)=3$ (see \cite{probm}).

\begin{theorem}
For $d>3$, $F_{d+1}\leq r(A_d)\leq (d^2-d+1)F_{d-1}F_{d-2}$.
\end{theorem}

\begin{proof}
For the lower bound, note that any rainbow solution $x=(x_1,\dots,x_d)$ of $A_d$ has $x_d\ge F_{d+1}$. So for $n=1$, $k\ge F_{d+1}$ and therefore $F_{d+1}\leq r(A_d)$.

For the upper bound, we compute an Ehrhart-like polynomial to refine the bounds in the proof of the \emph{(iii)$\implies$(ii)} part in Theorem \ref{thm:main}.

Consider the polytope $P=[0,1]^d\cap\ker(A_d)$.
The dimension of $\ker(A)$ is $2$, so $\dim(P)=2$. The following claim ensures that this is a triangle.

\begin{claim}
The vertices of $P=[0,1]^d\cap\ker(A_d)$ are $O=(0,0,\dots,0)$, $A=\frac{1}{F_{d-1}}(F_0,F_1,\dots,F_{d-1})$ and $B=\frac{1}{F_{d-2}}(1,F_0,\dots,F_{d-2})$.
\end{claim}

\begin{proof}
Suppose that $v=(v_1,\dots,v_d)$ is a vertex of $P$. Since $\dim(P)=2$ then two entries of $v$ are equal to either $0$ or $1$.
If both entries are $0$ then $v$ is the origin $O$. So assume $v$ contains an entry $v_i=1$.

Because $v$ is a Fibonacci sequence containing nonnegative elements then $v_j\le v_{j+1}$ for $j\ge 2$.
Consequently, if $v_i=1$, then $i\in\{1,d-1,d\}$, as otherwise $v_{i+2}\ge 2$.

If $v_1=1$ then $v_2=0$, as otherwise $v_3>1$.
In this case $v=(1,0,1,1,2,\dots)$, which is $B$ if $d=4$ and invalid otherwise.

If $v_{d-1}=1$ then $v_d=1$, as $v_d\ge v_{d-1}$.
In this case $v=(\dots,-1,1,0,1,1)$, which is $B$ if $d=4$ and invalid otherwise.

So assume $v_d=1$ and $v_1,v_{d-1}\neq 1$. No other entry of $v$ can be $1$, so some entry $v_j=0$.
If $j\ge 2$ then $v_2=0$ because $v_2\le v_j$, in this case $v=B$.
If $v_1=0$ then $v=A$.
\end{proof}

Therefore the dilation $(F_{d-1}F_{d-2})P$ is an integer polytope. Now we need to count the number $L_d(t)$ of lattice points in $(tF_{d-1}F_{d-2})P$ with positive entries.

Let $Q$ be the polytope obtained by projecting $(F_{d-1}F_{d-2})P$ onto its first two coordinates. It is a triangle with vertices $(0,0)$, $(F_{d-1},0)$, and $(0,F_{d-2})$.

\begin{claim}
For any $t\in\Z$, this projection gives a bijection between the lattice points in $(tF_{d-1}F_{d-2})P$ and the lattice points in $tQ$.
\end{claim}

\begin{proof}
Clearly the projection is injective.
Suppose $(a,b)$ is an integer point in the dilation $nQ$. 
Then $a,b\ge 0$ and $F_{d-2}a+F_{d-1}b\leq tF_{d-1}F_{d-2}$. Consider the Fibonacci sequence
$$(a,b,a+b,\dots,F_{d-2}a+F_{d-1}b).$$
This sequence is contained in $(tF_{d-1}F_{d-2})P$ and is projected to $(a,b)$.
\end{proof}

We may compute $L_d$ by counting the lattice points in $tQ$ with positive entries:
\begin{align*}
L_d(t)&=\frac{(tF_{d-1}+1)(tF_{d-2}+1)+(t+1)}2-(tF_{d-1}+1)-(tF_{d-2}+1)+1\\
&=\frac{F_{d-1}F_{d-2}}2t^2-\frac{F_d-1}{2}t.
\end{align*}

Now we need to show that if $k=(d^2-d+1)F_{d-1}F_{d-2}$ then for any equinumerous $k$-coloring of $[kn]$ there is a rainbow solution to $A_dx=0$.
From the computation in \eqref{eq:bound}, we know there is a solution whenever
\begin{align*}
L_d\left(\frac{kn}{F_{d-1}F_{d-2}}\right)&=L_d\left((d^2-d+1)n\right)\\
&=\frac{F_{d-1}F_{d-2}}2(d^2-d+1)^2n^2-\frac{F_d-1}{2}(d^2-d+1)n\\
&>\frac{d(d-1)}{2k}(kn)^2=\frac{F_{d-1}F_{d-2}}{2}d(d-1)(d^2-d+1)n^2.
\end{align*}
This is equivalent to
$F_{d-1}F_{d-2}n>F_d-1$,
which is true for all integers $d>3$ and $n\ge 1$.
\end{proof}

\section*{Acknowledgments}

We are grateful to UC MEXUS for the support received funding this research collaboration. We are also indebted to UC Davis and CINNMA for the hospitality and support received. The second author was supported by NSF grant DMS-0135345, the third author was supported by PAPIIT IA102013, the forth author was supported by PAPIIT IN101912 and the last three authors were supported by CONACyT project 166306. Last but not least we wish to thank Prof. Jacob Fox and Prof. Leslie Hogben for their helpful comments and suggestions.

\end{document}